\documentclass[reqno, 11pt, a4paper]{amsart}

\usepackage{
    a4wide,
    amssymb,
    bbm,
    centernot,
    mathtools,
} 
\usepackage{graphicx}
\graphicspath{{./images/}}
\usepackage[cal=euler, scr=boondoxo]{mathalfa}
\usepackage[colorlinks=true, linkcolor=blue,citecolor=blue]{hyperref}
\usepackage{float}
\usepackage{comment}

\usepackage{caption}
\usepackage{subcaption}

\usepackage{graphicx}
\usepackage{amssymb}
\usepackage{amsthm}
\usepackage{tikz}

\usepackage{centernot}
\usepackage{mathtools}
\usepackage{ stmaryrd }

\usepackage{hyperref}
\hypersetup{allcolors=blue,
pdftitle={NV Scheme to SLE},
    pdfpagemode=FullScreen}
\numberwithin{equation}{section}
\usepackage{bbm}

\usepackage{braket}
\usepackage{amsmath,amsthm,amssymb}
\usepackage{physics}
\usepackage{mathtools}
\usepackage{bm}
\usepackage{esvect}
\usepackage[english]{babel}
\usepackage{extarrows}

\usepackage{setspace}
\usepackage{amsmath}

\numberwithin{equation}{section}
\usepackage{bbm}

\usepackage{braket}
\usepackage{amsmath,amsthm,amssymb}
\usepackage{physics}
\usepackage{mathtools}
\usepackage{bm}
\usepackage{esvect}
\usepackage[english]{babel}
\usepackage{xcolor}

\newtheorem{theorem}{Theorem}[section]
\newtheorem{lemma}[theorem]{Lemma}

\newtheorem{corollary}[theorem]{Corollary}
\newtheorem{proposition}[theorem]{Proposition}

\theoremstyle{definition}
\newtheorem{definition}[theorem]{Definition}

\theoremstyle{remark}
\newtheorem{remark}[theorem]{Remark}

\numberwithin{equation}{section}

\usepackage{tikz}

\begin{document}

\definecolor{airforceblue}{RGB}{204, 0, 102}
\newenvironment{draft}
  {\par\medskip
  \color{airforceblue}%
  \medskip}


\title[Ninomiya--Victoir splitting algorithm to SLE]{Splitting algorithm and normed convergence for drawing the random fractal Loewner curves}



\author{Jiaming Chen}
\address{Courant Institute \& NYU--ECNU Institute of Mathematical Sciences, New York University}
\curraddr{251 Mercer Street
New York, NY 10012, United States}
\email{chen.jiaming@cims.nyu.edu}
\thanks{}

\author{Vlad Margarint}
\address{Department of Mathematics \& Statistics, University of North Carolina at Charlotte}
\curraddr{9201 University City Blvd, NC 28223, Charlotte, United States}
\email{vmargari@uncc.edu}
\thanks{}

\begin{abstract}
In the first part of the paper we propose and study the approximation of the $SLE_\kappa$ trace via the Ninomiya--Victoir splitting algorithm. We prove the uniform convergence in probability with respect to the sup-norm to the distance between the $SLE_\kappa$ trace and the output of the Ninomiya--Victoir splitting algorithm when applied in the context of the Loewner differential equation. {Further investigations on the $L^p$-norm convergence is also exhibited, shedding light on the more delicate convergence structure.} In the second part we show the uniform convergence of the approximation of the $SLE_\kappa$ trace obtained using a different scheme that is based on the linear interpolation of the Brownian driving force.
\end{abstract}

\dedicatory{}
\maketitle

\tableofcontents

\section{Introduction}
  The Loewner equation was introduced by Charles Loewner in $1923$ and it was one of the important ingredients in the proof of the Bieberbach Conjecture that was done by Louis de Branges, years later in $1985$.
In $2000$, Oded Schramm introduced a stochastic version of the Loewner equation. The stochastic version of the Loewner evolution, i.e. the Schramm--Loewner evolution, $SLE_\kappa$, generates a one parameter family of random fractal curves that are proved to describe scaling limits of a number of discrete models that appear in two-dimensional statistical physics.
For example, in (\cite{LS}, $Sec.\;1.1$) it is shown that the scaling limit of loop erased random walk, with the loops erased in a chronological order, converges in the scaling limit to $SLE_{\kappa}$ with  $\kappa = 2$. Moreover, other two dimensional discrete models from statistical mechanics including Ising model cluster boundaries, Gaussian free field interfaces, percolation on the triangular lattice at critical probability, and uniform spanning trees were proved to converge in the scaling limit to $SLE_{\kappa}$ for values of $\kappa=3$, $\kappa=4$, $\kappa=6$, and $\kappa=8$ respectively in the series of works \cite{SS}, \cite{SSch}, \cite{Sm}, \cite{LS}. \par

There are various versions of Loewner equations. One of them is the forward Loewner equations defined in the upper half-plane $\mathbb{H}$ and a fixed time interval $[0,T]$ 
\begin{equation}
\label{1.1}
\partial_t g_t(z)=\frac{2}{g_t(z)-\lambda(t)}
\end{equation} 
with initial condition $g_0(z)=z,\;\text{for all}\;z\in\mathbb{H}$ and the continuous driving force $\lambda:[0,T]\to\mathbb{R}$.

The family of maps $(g_t)_{0\leq t\leq T}$ is called the forward Loewner chain. For all $z\in\mathbb{H}$, the solution of the above forward Loewner equation is uniquely defined up to $T_z=\inf\{t\geq0,\,g_t(z)=\lambda(t)\}$. Over time, the hulls, that is the sets $K_t=\{z\in\mathbb{H},\,T_z\leq t\}$ grow. It is also known that for all $t\in[0,T],$ there is a unique conformal map $\,g_t:\mathbb{H}\backslash K_t\to\mathbb{H}$ satisfying the hydrodynamic normalization
\begin{equation*}
\lim\limits_{z\to\infty}[g_t(z)-z]=0.
\end{equation*}
We study $(g_t)_{0\leq t\leq T}$ parametrized by upper half-plane capacity 
\begin{equation*}
g_t(z)=z+\frac{2t}{z}+o(1/\abs{z}),\;\textit{as}\;\abs{z}\to\infty,
\end{equation*}
where by (\cite{Antti}, $Lem.\;4.1$), the coefficient of the $z$-term is $1$ and each coefficient $a_k$ of the term $z^{-k},\;k\in\mathbb{N}_+$ is real.\par
We are particularly interested in the case when the Loewner chain $(g_t)_{0\leq t\leq T}$ is generated by a curve $\gamma:[0,T]\to{\overline{\mathbb{H}}}$. By (\cite{latexcompanion}, $Thm.\;4.1$), this is equivalent to the existence and continuity in $0<t<T$ of 
\begin{equation*}
\gamma(t)=\lim\limits_{\epsilon\to 0^+}g_t^{-1}(\lambda(t)+i\epsilon).    
\end{equation*}
Then, for each $t \in [0, T].$ the domain of the map $g_t(z)$ is the unbounded component of $\mathbb{H}\backslash\gamma([0,t])$.

Throughout our paper, we work with stochastic Loewner chains and for this we introduce a standard probability space $(\Omega,\mathcal{F},\mathbb{P})$, where the Brownian motion is defined. Furthermore,  for $\kappa \in \mathbb{R}_+$, consider the forward  Loewner chain driven by Brownian motion $\sqrt{\kappa}B_t,\;\text{with}\;t\in[0,1]$
\begin{equation}
\label{1.5}
\partial_t g_t(z)=\frac{2}{g_t(z)-\sqrt{\kappa}B_t}
\end{equation} 
with $g_0(z)=z.$ It was shown in \cite{latexcompanion} for $\kappa \neq 8,$ and in \cite{LS} for $\kappa=8$ that the Loewner chains driven by Brownian motion are generated by a trace, that is, the following limit exists and is continuous in time a.s.
\begin{equation*}
\gamma(t)=\lim\limits_{y\to0^+}g_t^{-1}(\sqrt{\kappa}B_t+iy).
\end{equation*}  
We call $\gamma(t)$ the $SLE_{\kappa}$ trace.

  Throughout the proof we will use the backward Loewner equation which is related to the forward version $(\ref{1.1})$ and admits the form
    \begin{equation}
    \label{1.7}
        \partial_th_t(z)=\frac{-2}{h_t(z)-\sqrt{\kappa}(B_{1-t}-B_1)}
    \end{equation}
    with $h_0(z)=z$ for all $z\in\mathbb{H}$. This backward Loewner equation generates a curve $\eta:[0,1]\to\overline{\mathbb{H}}$. Notice that $g_t(z)$ and $h_t(z)$ are both driven by Brownian motions with different time directions. In fact, there is a correspondence (\cite{Zhan}, $Sec.\;1.1$) between the two evolutions.\par
     The random set $\eta([0,1])$ has the same law as the chordal $SLE_\kappa$ trace $\gamma([0,1])$ modulo a real scalar shift $\sqrt{\kappa}B_{T=1}(\omega)$ for $\omega\in\Omega$. For convenience, we call $\eta(t)$ the shifted Loewner curve. We will simulate the Loewner curve $\eta(t)$ instead of $\gamma(t)$ for convenience, since the former curve preserves statistical properties of the latter curve.\par
      Numerical schemes for the $SLE_{\kappa}$ are of significant interest in applied probability, especially when they shed light on the scaling limits of planar Statistical Mechanics models. Simulating the $SLE_{\kappa}$ traces is a complicated problem because its study requires the analysis of singular diffusion. Obtaining good numerical schemes for the $SLE_{\kappa}$ is of significant interest both in the Numerical Analysis literature, where $SLE_{\kappa}$ is an important example of a highly fractal curve, and in the Random Geometry literature, where simulations of $SLE_{\kappa}$ can be used both for illustrations, to gain intuition about the curve, and for numerical experiments to understand the curve’s properties.\par
      In the last years, results on numerical schemes to approximate the $SLE_\kappa$ trace, along with mathematical proofs of their convergence, appeared in the body of literature on the topic. For example, V. Beffara used a Euler Scheme to produce approximations of the $SLE_\kappa$ hulls.  Marshall and Rohde \cite{tom2} introduced a different method that square-root interpolates the driver. We refer the reader to \cite{tom},\cite{tomk} for a detailed description and implementation of this method. In addition, the convergence of the square-root interpolation  of the driver method is studied in detail in \cite{Tran}. The above two numerical simulation schemes are also expected to apply to other singular diffusion such as multiple Schramm--Loewner evolution, see \cite{Bauer,Chen/Laulin}. In this context, the stability of the evolution under perturbation is worth considering, see also \cite{Chen}. {Different numerical methods essentially have different advantages. And we focus on a particular method called the \textit{splitting algorithm} \cite{Buckwar}.}

  In this paper, we study two schemes to simulate the $SLE_\kappa$ trace and we provide theoretical bounds of their convergence. The first one is the Ninomiya--Victoir {splitting algorithm} and the second one is obtained from the linear-interpolation of the sample paths of the Brownian driving force. {The scheme of Ninomiya--Victoir splitting algorithm has been applied to many models in the study of stochastic differential equations \cite{Ableidinger}. This splitting-type algorithm is manifested because preserve important structural properties of the stochastic dynamics \cite{Buckwar}. First, it is hypoelliptic in every step \cite{Petersen} and it is geometrically ergodic \cite{Mattingly}. Moreover, it preserves oscillatory dynamics even at large time iteration steps \cite{Buckwar2}.} The {idea of using} Ninomiya--Victoir splitting {algorithm} in the context of simulating the $SLE_\kappa$ trace was first considered in \cite{Foster}, where the algorithm was performed on a compact time horizon and its error analysis was discussed. {Due to the analytically tractable vector fields of the Loewner equation, the Ninomiya-Victoir splitting algorithm is particularly well suited for $SLE_\kappa$ simulation \cite{Foster}.} This {numerical method} is important in the Numerical Analysis literature and is known to achieve a good weak convergence rate. We believe that the Ninomiya--Victoir {algorithm} is the first high order numerical method that is applied in the context of simulating $SLE_\kappa$ trace. {From the more practical perspective, we employ the adaptive step size, proposed in \cite{tomk}, to address the singularity at the origin of the Loewner dynamics.} Furthermore, it was shown in \cite{Foster} that this method preserves the second moment of the backward Loewner dynamics. In this work, we have shown that the scheme from the Ninomiya--Victoir algorithm converges to $SLE_\kappa$. Following \cite{Foster}, examples code for this method can be found at \cite{James}. {Applying the splitting algorithm to $SLE_\kappa$ is natural \cite{Foster} because such a method can be interpreted as the solution of an ODE/RDE governed by the same vector fields but driven by a piecewise linear path. A slightly different splitting-type method in the context of stochastic differential equations can also be found at \cite{Foster2}.}

The paper is divided into several sections, the first one being the introduction. In the second section, we introduce the {precise framework of the} Ninomiya--Victoir splitting {algorithm, so far an efficient scheme in simulating the $SLE_\kappa$ curves}. In the third section we prove the uniform convergence in probability of the Ninomiya--Victoir splitting {algorithm. {A quantitative discussion on the exact convergence rate is deferred to future work.} The next section is dedicated to the discussion of the more general convergence w.r.t. $L^p$ norms, whence we have various perspectives to view the convergence of the splitting algorithms.} In the last section, we discuss the linear-interpolation to the Brownian driving force. {This approximation is worth studying since it is in some ways the simplest continuous approximation of the Brownian motion.}\par
The most technical part in proving the uniform convergence in probability of Ninomiya--Victoir {algorithm} is Proposition $3.14$ in which we compare a perturbation term at finitely many deterministic time, which is the most difficult part, and {such technique} does not come from previous literature. Such comparison requires the analysis of a time-reversed Loewner flow. Besides, it is also worth mentioning that the proof of Proposition $3.16$ requires an explicit expression of the Ninomiya--Victoir {step}s. {Proposition 4.8 should also be mentioned since only in this assertion we invoke the interpolation relations of Lebesgue spaces as well as the maximum process of Brownian motions.}

\section{Ninomiya--Victoir splitting {algorithm}}
\label{sec:chapter 1}
In this section, we introduce the Ninomiya--Victoir Splitting {Algorithm}. Before that, we will rephrase the forward and backward Loewner evolutions in a convenient manner. If we set $\widehat{g}_t(z)\coloneqq g_t(z)-\sqrt{\kappa}B_t\,\text{for all}\,z\in\mathbb{H}\backslash K_t$, then the forward Loewner chain driven by Brownian motion can be rewritten as
\begin{equation}\begin{aligned}
\label{2.1}
d\widehat{g}_t(z) =\frac{2}{\widehat{g}_t(z)}\,dt-\sqrt{\kappa}\,dB_t\;\;\;\text{and}\;\;\;\widehat{g}_0(z) =z
\end{aligned}\end{equation}
with $z\in\mathbb{H}\backslash K_t$.

 Moreover, let us consider $Z_t(z):=h_t(z)-\sqrt{\kappa}B_t$, for all $z\in\mathbb{H}$ (see \cite{Foster}, $Eqn.\;6.5$). Then the backward Loewner differential equation driven by Brownian motion can be rewritten as
\begin{equation}\begin{aligned}
\label{2.2}
dZ_t=-\frac{2}{Z_t}\,dt+\sqrt{\kappa}\,dB_t\;\;\;\text{and}\;\;\;Z_0=iy.
\end{aligned}\end{equation}
For the initial condition, we consider $y>0$ to be taken sufficiently small.

We are now ready to introduce the Ninomiya--Victoir {algorithm}.

\begin{definition}{\textbf{Ninomiya--Victoir {Algorithm}}}\\
Consider $n$-dimensional SDE on $t\in[0,+\infty)$ with the form
\begin{equation*}\begin{aligned}
dW_t=L_0(W_t)\,dt+L_1(W_t)\,dB_t\;\;\;\text{and}\;\;\;W_0=\xi,
\end{aligned}\end{equation*}
where $\xi\in\mathbb{R}^n$ and $L_i:\mathbb{R}^n\to\mathbb{R}^n$ are smooth vector fields. For all $t\in\mathbb{R}_+$ and $x\in\mathbb{R}^n$, let the flow $\exp(tL_i)x,\,i=1,2$, denote the unique solution at time $u=1$ to the ODE
\begin{equation*}\begin{aligned}
\frac{dy}{du}=tL_i(y)\;\;\;\text{and}\;\;\;y(0)=x.
\end{aligned}\end{equation*}
For a fixed {algorithm} step $n\in\mathbb{N}$, we choose an arbitrary, possibly non-uniform, partition $\{t_0=0,t_1,\ldots,t_n=1\}$ with step-size $h_k=t_{k+1}-t_k$. We approximate a numerical solution $\{\widetilde{W}_{t_k}\}_{0\leq k\leq n}$ in the sense that $\widetilde{W}_0=\xi$ and 
\begin{equation*}
\widetilde{W}_{t_{k+1}}=\exp(\frac{1}{2}h_kL_0)\exp\bigg(B_{t_k,t_{k+1}}L_1\bigg)\exp(\frac{1}{2}h_kL_0)\widetilde{W}_{t_k}
\end{equation*}
for all $k=0,1,\ldots,n$ and where $B_{t_k,t_{k+1}}$ is the abbreviation for $B_{t_{k+1}}-B_{t_k}$. In fact, the approximation $\{\widetilde{W}_{t_k}\}_{0\leq k\leq n}$ enjoys an integral form between every two discretization points
\begin{equation}
\label{2.6}
\widetilde{W}_t=\xi+\frac{1}{2}\displaystyle\int_0^tL_0(\widetilde{W}_s^{(2)})\,ds+\int_0^tL_1(\widetilde{W}_s^{(1)})\,dB_s+\frac{1}{2}\int_0^tL_0(\widetilde{W}_s^{(0)})\,ds,  
\end{equation}
where the above three discretization processes defined on each time interval $[t_k,t_{k+1}]$ admit the form
\begin{equation*}\begin{aligned}
\widetilde{W}^{(0)}_t&\coloneqq\exp(\frac{1}{2}(t-t_k)L_0)\widetilde{W}_{t_k},\\
\widetilde{W}^{(1)}_t&\coloneqq\exp\bigg(B_{t_k,t}L_1\bigg)\widetilde{W}^{(0)}_{t_{k+1}},\\
\widetilde{W}^{(2)}_t&\coloneqq\exp(\frac{1}{2}(t-t_k)L_0)\widetilde{W}^{(1)}_{t_{k+1}}.
\end{aligned}\end{equation*}
\end{definition}
Looking back at our backward Loewner equation, we write $L_0(z)=-2/z$ and  $L_1(z)=\sqrt{\kappa}$. Hence, by (\cite{Foster}, $Thm.\;6.2
$) the following form is immediate
\begin{equation}\begin{aligned}
\label{2.8}
\exp(tL_0)z=\sqrt{z^2-4t}\;\;\;\text{and}\;\;\;
\exp(tL_1)z=z+\sqrt{\kappa}t.
\end{aligned}\end{equation}
Given the above arbitrary partition $\{t_0=0,t_1,\ldots,t_n=1\}$, we could formulate an approximated solution $\{\widetilde{Z}_{t_k(t)}\}_{0\leq k\leq n}$ via the Ninomiya--Victoir splitting scheme in the following
\begin{equation}
\label{2.9}
\widetilde{Z}_{t_{k+1}}\coloneqq\sqrt{\bigg(\sqrt{\widetilde{Z}_{t_k}^2-2{h_k}}+\sqrt{\kappa}B_{t_k,t_{k+1}}\bigg)^2-2h_k}
\end{equation}
with the initial value $\widetilde{Z}_0=iy$ specified at each $n^{th}$ {step}.

\begin{figure}[H]
\begin{subfigure}{0.33\textwidth}
    \includegraphics[width=\textwidth]{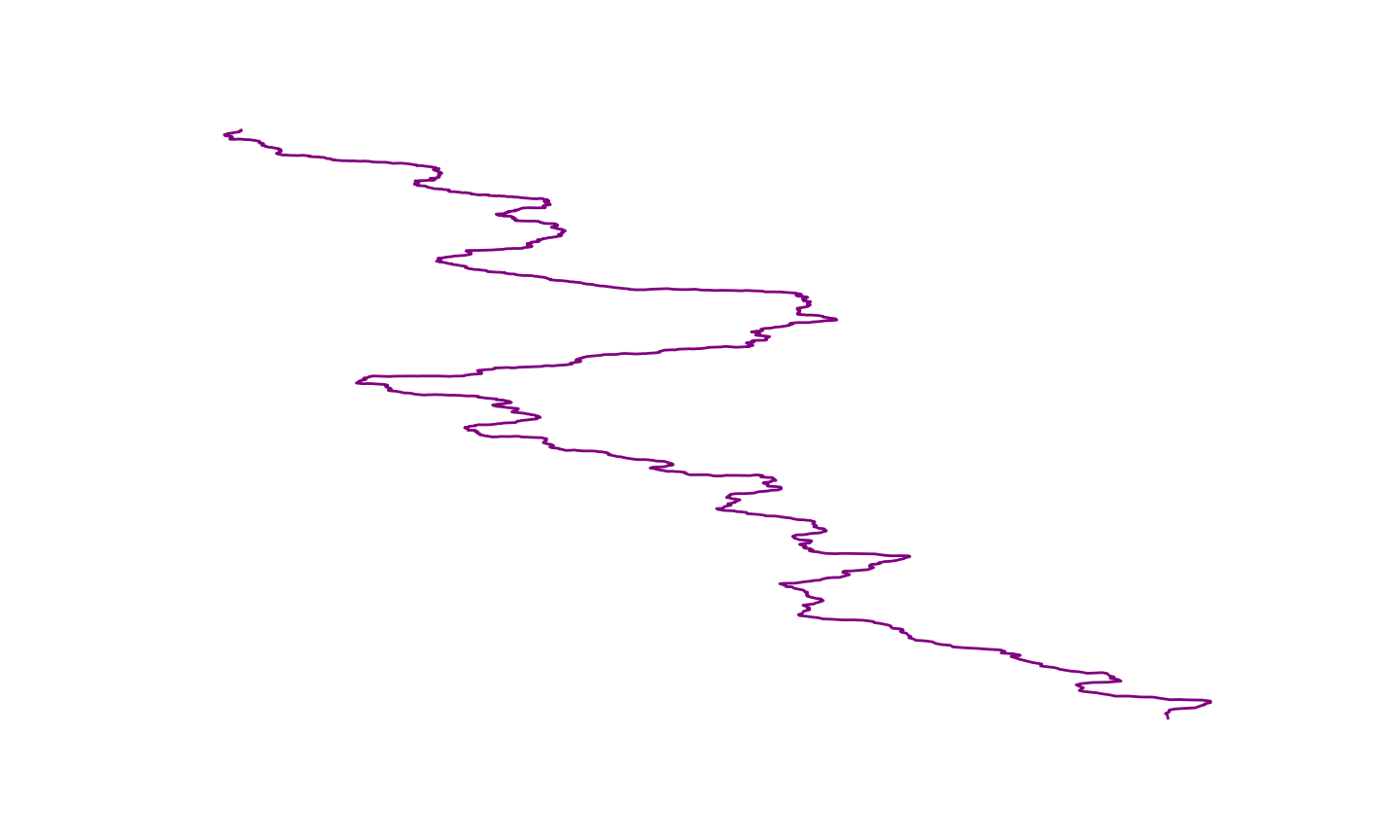}
    \label{NVkappa=1/2}
\end{subfigure}\hfill
\begin{subfigure}{0.33\textwidth}
    \includegraphics[width=\textwidth]{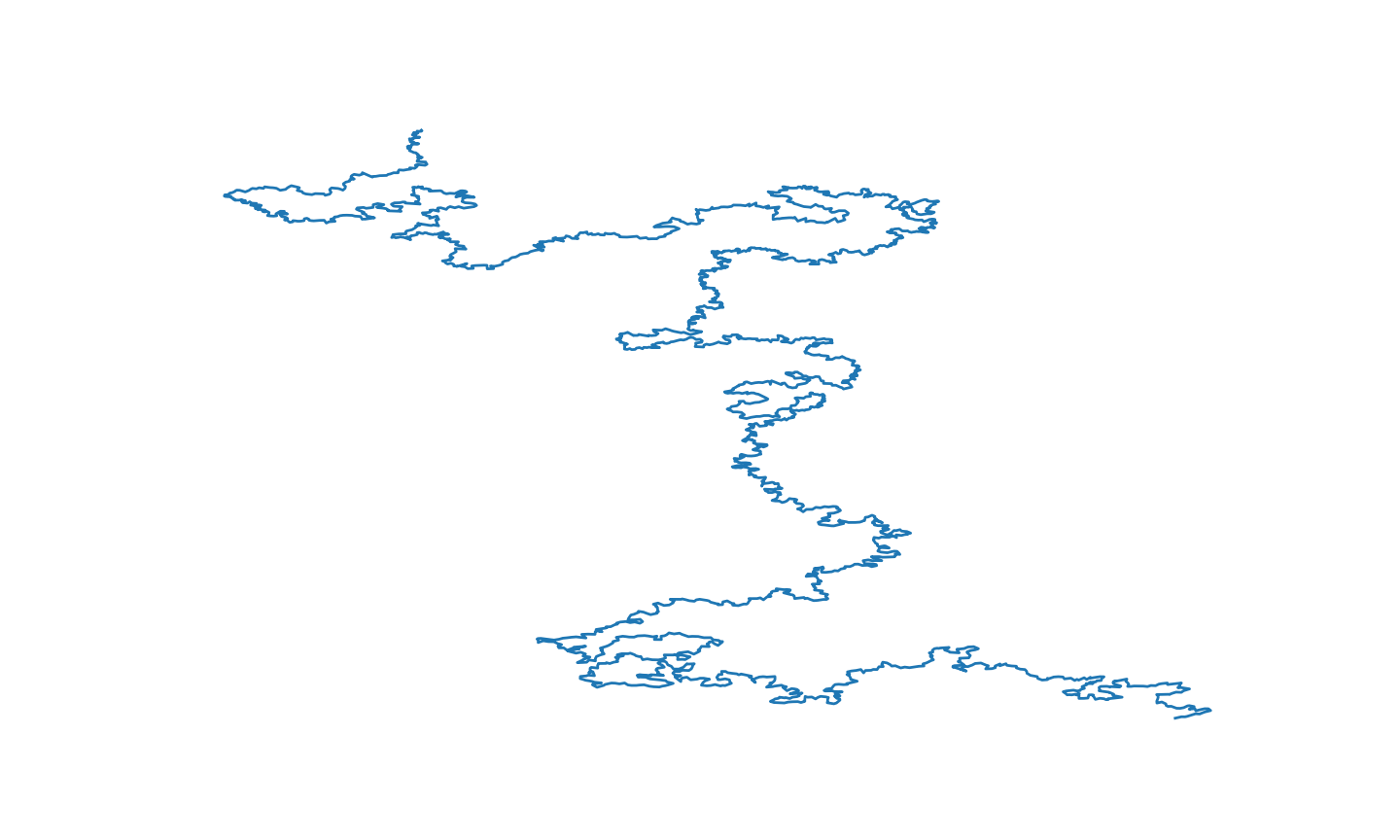}
    \label{NVkappa=8/3}
\end{subfigure}
\begin{subfigure}{0.33\textwidth}
    \includegraphics[width=\textwidth]{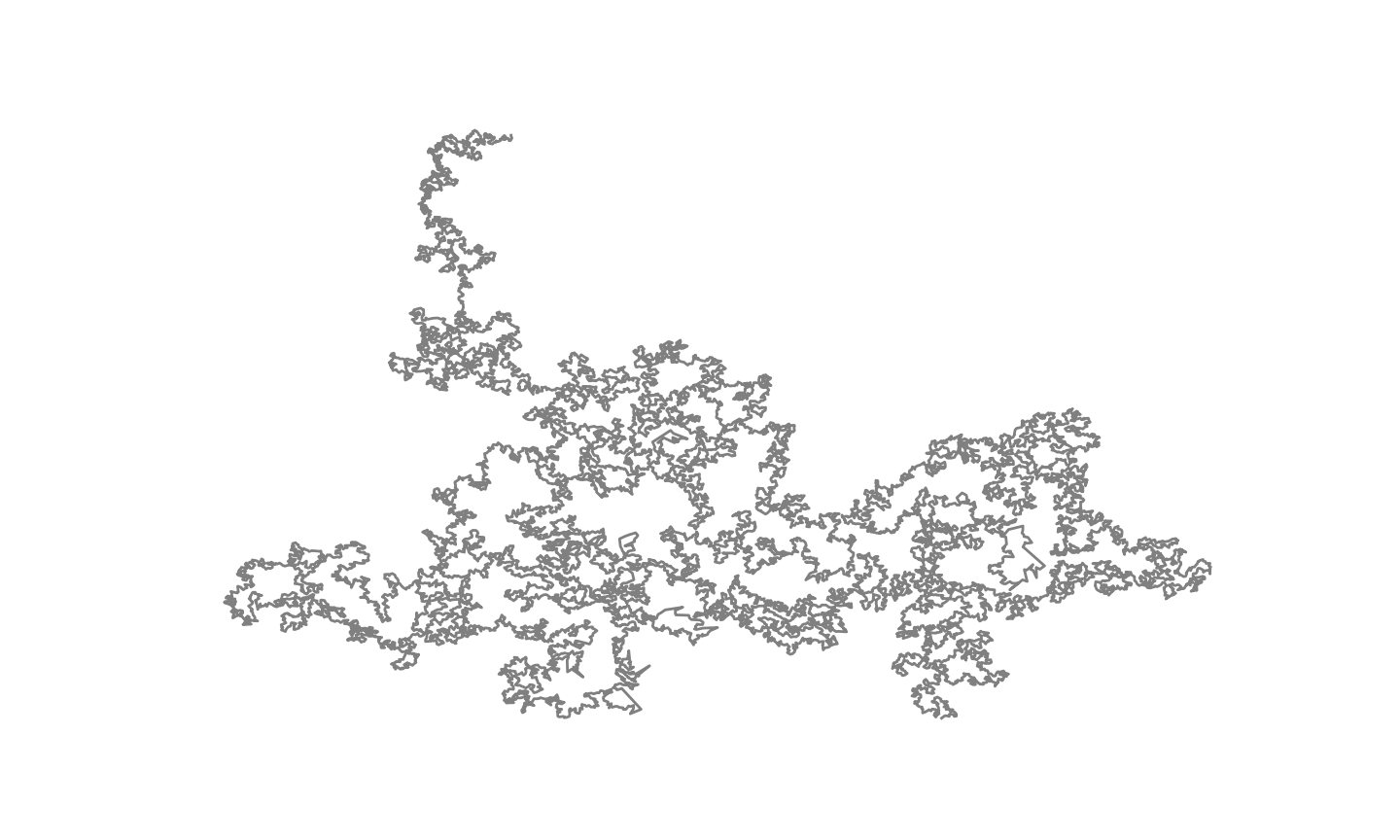}
    \label{NVkappa=6}
\end{subfigure}
\caption{Simulation of the $SLE_\kappa$ traces using the Ninomiya--Victoir splitting algorithm with different $\kappa$ and with {adaptive step size instead of fixed step size}. From left to right - $\kappa=1/2$, $\kappa=8/3$, $\kappa=6$. Credits to James Foster \cite{James}.For detailed pictures with scales we refer the reader to the Github page in \cite{James}.}
\label{SLE NV scheme}
\end{figure}

\section{{Sup-norm} convergence in probability}
\label{sec:chapter 2}
In the following, we use $\norm{\vdot}_{[0,1],\infty}$ for the sup-norm on the time interval $[0,1].$ In addition, we use $\norm{\vdot}$ to denote the mesh size of our partition for this time interval.
In this section we give a uniform convergence in probability to the decay rate of the $\norm{\vdot}_{[0,1],\infty}$ norm (\textit{i.e. sup-norm}) between the original Loewner curve and our simulation.

\begin{definition}
    Let $Z_t(iy_n)$ be the solution to $(\ref{1.5})$ started from $iy_n\in\mathbb{H}$ and $\widetilde{Z}_t(iy_n)$ be its approximation following Ninomiya--Victoir scheme, and let $\eta(t)$ be the shifted Loewner curve defined before.
\end{definition}

Notice that at each {algorithm} step $n\in\mathbb{N}$, we specify an initial condition $y\in\mathbb{R}_+$ and let the approximated sample paths $(\widetilde{Z}_t(iy))_{0\leq t\leq1}$ evolve according to the backward Loewner equation $(\ref{1.7})$. To ensure a $\norm{\vdot}_{[0,1],\infty}$ convergence result, we not only require the mesh of the partition tends to $0$, but also choose a sequence $\{y_n\}\subset\mathbb{R}_+$ so that $y_n\to0^+$ strictly and monotonically.
\begin{remark}
Notice that we cannot let $y_n\equiv y$ for some $y>0$, otherwise the convergence pattern breaks down and hence strict monotonicity of $\{y_n\}$ is necessary. On the other hand, the decay rate of $\{y_n\}$ should not be too fast to destroy the probability inequality w.r.t. the $\norm{\vdot}_{[0,1],\infty}$ norm, which will be seen in the following context.
\end{remark}
In this section, we manually set $y_n=n^{-1/2}\;\text{for all}\;n\in\mathbb{N}$. This choice of $\{y_n\}$ actually satisfies the requirements in $Rmk.\;3.2$ for the initial conditions. We consider $\mathcal{D}_n$ to be a uniform partition of $[0,1]$ with mesh-size $||\mathcal{D}_n||.$
\begin{definition}
For all $t\in[0,1]$, given an arbitrary uniform partition $\mathcal{D}_n$, we define $\{t_k(t),t_{k+1}(t)\}\subseteq\mathcal{D}_n$ to be the neighboring two points in the partition $\mathcal{D}_n$ between which $t$ resides, \textit{i.e.} $t_k(t)\leq t<t_{k+1}(t)$. 

\end{definition}
We emphasize that the index $k$ is changing as the parameter $t$ evolves in time. We prefer this notation so as to make it close to the notation introduced in the definition of the splitting scheme. 

We are now reaching our main object: an upper bound for the probability of $\norm{\vdot}_{[0,1],\infty}$ norm of $(\eta(t)-\widetilde{Z}_t(iy_n))_{0\leq t\leq1}$ to be small in the sense given by the following theorem.
\begin{theorem}
Let $\eta(t)$ be the backward $SLE_\kappa$ trace for $\kappa \neq 8$.
There exist two non-increasing functions $\varphi_i:\mathbb{N}\to\mathbb{R}_+$ so that $\lim_{n\to\infty}\varphi_i(n)=0^+\;\text{with}\;i=1,2$. If the mesh $||\mathcal{D}_n||\to0^+$ with $n\to\infty$, at a rate  $||\mathcal{D}_n||=o(n^{-3})$, then
\begin{equation*}
\mathbb{P}\bigg(\norm{\eta(t)-\widetilde{Z}_t(iy_n)}_{[0,1],\infty}\leq\varphi_1(n)\bigg)\geq1-\varphi_2(n).
\end{equation*}
\end{theorem}
\begin{remark}
   The exact form of decreasing functions $\varphi_1$ and $\varphi_2$ can be found later in (\ref{result1}) and (\ref{result2}). Therefore, we have verified the uniform convergence in probability of the Ninomiya-Victoir {splitting algorithm} in the context of $SLE_\kappa$. {For a more quantitative description of the exact convergence rate, we conceive that an analysis of the more delicate structure of $SLE_\kappa$ is required, which can for instance be found in \cite{Tran}. We plan to work out this more quantitative part in future projects.}
\end{remark}

To prove this theorem, we use
\begin{equation}\begin{aligned}
\label{3.2a}
    &\abs{Z_t(iy_n)-\widetilde{Z}_t(iy_n)}\leq\abs{Z_t(iy_n)-Z_{t_k(t)}(iy_n)}\\&\;\;\;\;\;\;\;\;\;\;\;\;\;\;\;\;\;\;\;+\abs{Z_{t_k(t)}(iy_n)-\widetilde{Z}_{t_k(t)}(iy_n)}+\abs{\widetilde{Z}_{t_k(t)}(iy_n)-\widetilde{Z}_t(iy_n)}
\end{aligned}\end{equation}
in order to estimate
\begin{equation*}
\abs{\eta(t)-\widetilde{Z}_t(iy_n)}\leq\abs{\eta(t)-Z_t(iy_n)}+\abs{Z_t(iy_n)-\widetilde{Z}_t(iy_n)}.\;\;\:\:
\end{equation*}
The $(\ref{3.2a})$ follows from the lemma below.
\begin{lemma}{{\rm (\cite{Lawler}, $Thm.\;3.4.2$)}}
There exist $c_1,c_2>0$ such that if we consider the event
\begin{equation}
\label{3.3}
    E_{n,1}^\prime\coloneqq\bigg\{\text{osc}(\sqrt{\kappa}B_t,\frac{1}{n})\leq c_1\sqrt{\frac{\log(n)}{n}}\:\bigg\}
\end{equation}
then we have 
\begin{equation*}
    \mathbb{P}\big(E_{n,1}^\prime\big)\geq1-\frac{c_2}{n^2}.
\end{equation*}
\end{lemma}
\begin{lemma}{{\rm (\cite{Tran}, $Eqn.\;21.$)}}
There exist $c_3,c_4>0\;\text{and}\;\beta_1\in(0,1)$ such that if we consider the event 
\begin{equation}
\label{3.5}
    E_{n,1}^{\prime\prime}\coloneqq\bigg\{\abs{\partial_z\widehat{g}_t^{-1}(iv)}\leq c_3\vdot v^{-\beta_1}\;\text{for all}\;t\in[0,1]\;\text{and}\;v\in[0,\frac{1}{\sqrt{n}}]\bigg\}
\end{equation}
then we have
\begin{equation*}
    \mathbb{P}\big(E_{n,1}^{\prime\prime}\big)\geq1-\frac{c_4}{n^{c_3/2}}.
\end{equation*}
\end{lemma}
We have an estimate to the first term to $(\ref{3.2a})$ with the form
\begin{equation}\begin{aligned}
\label{3.7}
    \abs{Z_t(iy_n)-Z_{t_k(t)}(iy_n)}&\leq\abs{Z_t(iy_n)-\eta(t)}+\abs{Z_{t_k(t)}(iy_n)-\eta\big(t_k(t)\big)}+\abs{\eta(t)-\eta\big(t_k(t)\big)}.
\end{aligned}\end{equation}
To proceed our discussion, we remind our readers of the following definition.
\begin{definition}
A continuous function $\phi:\mathbb{R}_+\to\mathbb{R}_+$ is called a subpower function if it is non-decreasing and satisfies
\begin{equation*}
    \lim\limits_{x\to\infty}x^{-\nu}\phi(x)=0,\;\text{for all}\;\nu>0.
\end{equation*}
\end{definition}
\begin{remark}
A typical subpower function is $\phi(x)=(\log x)^\alpha,\;\text{for real}\;\alpha>0$.
\end{remark}
With the notion of a subpower function, we have the following result.
\begin{proposition}
Let $\beta_1 \in (0,1)$. There exists a subpower function $\phi:\mathbb{R}_+\to\mathbb{R}_+$ such that if we consider the event
\begin{equation*}
    E_{n,1}^*\coloneqq\bigg\{\norm{\eta(t)-\eta\big(t_k(t)\big)}_{[0,1],\infty}\leq\frac{2\phi(\sqrt{n})}{(1-\beta_1)n^{(1-\beta_1)/2}}\bigg\}
\end{equation*}
and if $||\mathcal{D}_n||\leq n^{-1}$, then
\begin{equation*}
    \mathbb{P}\big(E_{n,1}^*\big)\geq1-\frac{c_2}{n^2}-\frac{c_4}{n^{c_3/2}}.
\end{equation*}
\end{proposition}
\begin{proof}
In the proof we omit the bracket in $t_k(t)$ and simply write this term as $t_k$, which will be clear from the context. In the proof, we follow the statement in (\cite{Tran}, $Lem.\;2.5$) with some obvious changes in notations. Since $\eta([0,1])$ has identical distribution to $\gamma([0,1])$ modulo a scalar shift $\sqrt{\kappa}B_1$, it is immediate that $\mathbb{P}\big(E_{n,1}^*\big)$ is equal to the probability of the event with an expression which we substitute $\gamma(t)$ (\textit{resp.} $\gamma\big(t_k(t)\big)$) into $\eta(t)$ (\textit{resp.} $\eta\big(t_k(t)\big)$). By (\cite{Tran}, $Lem.\;2.5$) there exists a subpower function $\phi:\mathbb{R}_+\to\mathbb{R}_+$ so that on the event $E_{n,1}^\prime\cap E_{n,1}^{\prime\prime}$, provided $0\leq t-t_k(t)\leq n^{-1}\;\text{for all}\;t\in[0,1]$, we have
\begin{equation*}\begin{aligned}
    \abs{\gamma(t)-\gamma(t_k(t))}&\leq\phi(\sqrt{n})\bigg(\int_0^{n^{-1/2}}\abs{\partial_z\widehat{g}_t^{-1}(ir)}\,dr+\int_0^{n^{-1/2}}\abs{\partial_z\widehat{g}_t^{-1}(ir)}\,dr\bigg)\\&\leq\phi(\sqrt{n})\vdot\frac{2}{1-\beta_1}n^{-(1-\beta_1)/2}
\end{aligned}\end{equation*}
where $\beta_1 \in (0,1).$
Hence $\mathbb{P}\big(E_{n,1}^*\big)\geq\mathbb{P}\big(E_{n,1}^\prime\cap E_{n,1}^{\prime\prime}\big)$ and the conclusion follows.
\end{proof}
To finish the evaluation of $(\ref{3.7})$ and then estimate the first term in $(\ref{3.2a})$, we need the following result

\begin{proposition}
Let $\epsilon_0\in(0,1)$. If we choose $M_n=n^{(1-\epsilon_0)/4}$ and consider the following event
\begin{equation*}
    E_{n,1}^{**}\coloneqq\bigg\{\norm{Z_t(iy_n)-\eta(t)}_{[0,1],\infty}\leq M_n\vdot y_n^{1-\epsilon_0}=\frac{1}{n^{(1-\epsilon_0)/4}}\bigg\},
\end{equation*}
then there exists $\epsilon_n\to0^+$ monotonically such that
\begin{equation*}
    \mathbb{P}\big(E_{n,1}^{**}\big)\geq1-\epsilon_n.
\end{equation*}
\end{proposition}
\begin{proof}
It is stated in (\cite{Foster}, $Lem.\;6.7$) that there exists $\epsilon_0\in(0,1)$ so that \textit{almost surely} we have
\begin{equation}
\label{3.14}
    \sup\limits_{t\in[0,1]}\abs{Z_t(iy_n)-\eta(t)}\leq C_1^\prime(\omega)\vdot y_n^{1-\epsilon_0},
\end{equation}
where $C_1^\prime(\omega)$ is \textit{almost surely} finite. Guaranteed with the existence of at least one $C_1^\prime(\omega)\in\mathbb{R}_+$ for almost all $\omega\in\Omega$, we define the collection $\mathcal{A}(\omega)\subseteq\mathbb{R}_+$ for those $\omega\in\Omega$ with which there exists at least one $C_1^\prime(\omega)$ satisfying $(\ref{3.14})$. Notice that the collection $\mathcal{A}(\omega)$ is defined except for a \textit{measure-zero} event. The well-ordering principle tells us that $\mathcal{A}(\omega)$ has a lower bound. Hence it is legitimate to define 
\begin{equation*}
    C_1(\omega)\coloneqq\inf\mathcal{A}(\omega),
\end{equation*}
which is \textit{almost surely} defined. Therefore we could simply assume $C(\omega)$ exists and is finite everywhere via subtracting a \textit{measure-zero} event from $\Omega$.\par
{
    Observe that $\mathbb{P}(E_{n,1}^{**})\geq\mathbb{P}(E_{n,1}^{**}\cap\{C_1(\omega)\leq M_n\})$. By (\ref{3.14}), we know that $\{C_1(\omega)\leq M_n\}\subseteq E_{n,1}^{**}$. Since by our choice, we have $M_n\to\infty$ monotonically, by setting $\epsilon_n\coloneqq\mathbb{P}(C_1(\omega)>M_n)$, the assertion is therefore verified.
}
\end{proof}
We have now discussed every term in $(\ref{3.7})$, it is time to finalize the estimate of the first term in $(\ref{3.2a})$.
\begin{proposition}
Let $\beta_1 \in (0,1)$. Given the assumptions that $||\mathcal{D}_n||\leq n^{-1}$ and $y_n=n^{-1/2}$, if we define the event
\begin{equation*}
    E_{n,1}\coloneqq\bigg\{\norm{Z_t(iy_n)-Z_{t_k(t)}(iy_n)}_{[0,1],\infty}\leq\frac{2\phi(\sqrt{n})}{(1-\beta_1)n^{(1-\beta_1)/2}}+\frac{2}{n^{(1-\epsilon_0)/4}}\bigg\},
\end{equation*}
then the following inequality holds
\begin{equation*}
    \mathbb{P}\big(E_{n,1}\big)\geq1-\frac{c_2}{n^2}-\frac{c_4}{n^{c_3/2}}-2\epsilon_n.
\end{equation*}
\end{proposition}
\begin{proof}
 On the event $E_{n,1}^*\cap E_{n,1}^{**}$, we know that for $\beta_1 \in (0,1)$, we have
 \begin{equation*}
     \sup\limits_{t\in[0,1]}\abs{Z_t(iy_n)-\eta(t)}\leq\frac{1}{n^{(1-\epsilon_0)/4}}
  \end{equation*} 
  and
  \begin{equation*}
     \sup\limits_{t\in[0,1]}\abs{\eta(t)-\eta\big(t_k(t)\big)}\leq\frac{2\phi(\sqrt{n})}{(1-\beta_1)n^{(1-\beta_1)/2}}.
 \end{equation*}
 Back to $(\ref{3.7})$, we see $\mathbb{P}\big(E_{n,1}\big)\geq\mathbb{P}\big(E_{n,1}^*\cap E_{n,1}^{**}\big)\geq1-c_2n^{-2}-c_4n^{-c_3/2}-2\epsilon_n$.
\end{proof}
Hence we have estimated the sup-norm of the first term in $(\ref{3.2a})$. This is in fact the most complicated term among these three terms. Next, we will estimate the sup-norm of the second term.\par
Inspecting $(\ref{2.9})$, we observe that the evolution $\widetilde{Z}_{t_{k}}\mapsto\widetilde{Z}_{t_{k+1}}$ resembles a Loewner map driven by constant forces on the sub-interval $[t_k,t_{k+1}]$. In fact, this is the case. We are going to split the total time interval $[0,1]$ into time sub-intervals $[t_k,t_{k+1}]$. On each time sub-interval, the evolution $\widetilde{Z}_{t_{k}}\mapsto\widetilde{Z}_{t_{k+1}}$ is a composition of
two local backward Loewner maps driven by constant forces with an intermediate parallel translation.\par
\begin{lemma}{{\rm (\cite{Fredrikk}, $Sec.\;2.$)}}
Given a constant driving force $t\mapsto A$ on the time interval $[0,T]$, the forward Loewner chain admits the form
\begin{equation*}
    g_t(z)=A+\big[(z-A)^2+4t\big]^\frac{1}{2}.
\end{equation*}
And this forward Loewner chain induces a time-reversed (\textit{i.e.} backward) Loewner chain at the final moment $t=T$ with the form
\begin{equation}
\label{3.25}
    h_T(z)=A+\big[(z-A)^2-4T\big]^\frac{1}{2}.
\end{equation}
\begin{proof}
We know $g_T(z)\circ h_T(z)=z$, for all $z\in\mathbb{H}$ by (\cite{Antti}, $Lem.\;4.10$). The result immediately follows.
\end{proof}
\end{lemma}
\begin{lemma}
    On each time sub-interval $[t_k,t_{k+1}]$, we consider the constant force $t\mapsto0$ on time interval $[t_k,t_k+\frac{h_k}{2}]$ and the constant force given by the corresponding value of the Brownian path at the final moment of the time sub-interval on  $[t_k+\frac{h_k}{2},t_{k+1}]$. We denote the backward Loewner chain driven by these constant forces as $\iota_{k,1}^\prime$ and $\iota_{k,1}^{\prime\prime}$, respectively. Consider the parallel translation $z\xmapsto{\iota_{k,2}}z+\sqrt{\kappa}B_{t_k,t_{k+1}}$. Then we have the composition
    \begin{equation*}
        \widetilde{Z}_{t_{k+1}}(iy_n)=\iota_{k,1}^{\prime\prime}\circ\iota_{k,2}\circ\iota_{k,1}^\prime\widetilde{Z}_{t_{k}}(iy_n).
    \end{equation*}
\end{lemma}
\begin{proof}
Inspect $(\ref{2.9})$ and $(\ref{3.25})$ and the conclusion follows.
\end{proof}
Then we come to the following proposition.
\begin{proposition}
If we consider the perturbation event
\begin{equation*}
    E_{n,2}\coloneqq\bigg\{\norm{Z_{t_k(t)}(iy_n)-\widetilde{Z}_{t_k(t)}(iy_n)}_{[0,1],\infty}\leq\frac{1}{\sqrt{4n+1}}\bigg\}
\end{equation*}
and if we  further restrict $||\mathcal{D}_n||\leq n^{-1}\wedge (4n+1)^{-3}$, then
\begin{equation*}
    \mathbb{P}\big(E_{n,2}\big)\geq1-2e^{-(4n+1)/\kappa}.
\end{equation*}
\end{proposition}
\begin{proof}
  From $Sec.\;2.$ we already know $Z_t(iy_n)=h_t(iy_n)-\sqrt{\kappa}B_t$. If we further choose $\widehat{B}_t\coloneqq B_{1-t}-B_1$, then $(\ref{2.2})$ can be written as
  \begin{equation*}
      \partial_th_t(z)=\frac{-2}{h_t(z)-\sqrt{\kappa}\widehat{B}_t}
  \end{equation*}
  with $h_0(z)=z$ and where $\widehat{B}_t$ has the law of a standard Brownian motion. We also comment that our splitting  scheme could be formulated in a similar fashion. Define the driving force
  \begin{equation*}
      \widetilde{\lambda}(t)\coloneqq0\vdot\mathbbm{1}_{[0,\frac{h_1}{2})}+\sum\limits_{k\geq1}\sqrt{\kappa}B_{t_k(t)}\vdot\mathbbm{1}_{[t_k(t)-\frac{h_k}{2},t_k(t)+\frac{h_k}{2}\wedge1)}.
  \end{equation*}
  The random process $\widetilde{\lambda}(t)$ can be viewed as a step-function interpolation to the sample paths of Brownian motion $\sqrt{\kappa}B_t$ on $[0,1]$. In this regard, we denote by $\widetilde{Z}_t^*(iy_n)$ the trajectory driven by the above driving force similar to $Z_t(iy_n)$ being driven by $\sqrt{\kappa}\widehat{B}_t$ in the following sense
  \begin{equation*}
      \widetilde{Z}_t^*(iy_n)=\widetilde{h}_t(iy_n)-\widetilde{\lambda}(t)
  \end{equation*}
  due to $Lem.\;3.13$ and that $(\widetilde{h}_t)_{t\in[0,1]}$ is a backward Loewner chain constrained by the following differential equation
  \begin{equation}
  \label{3.32}
      \partial_t\widetilde{h}_t(z)=\frac{-2}{\widetilde{h}_t(z)-\widehat{\lambda}_t}
  \end{equation}
  with $\widetilde{h}_0(z)=z$ and $\widehat{\lambda}_t\coloneqq\widetilde{\lambda}(1-t)-\widetilde{\lambda}(1)$. The above scheme brings us some consistency to the perturbation term $Z_t(iy_n)-\widetilde{Z}_t^*(iy_n)$. And our first goal is to estimate
  \begin{equation*}
      \abs{Z_t(iy_n)-\widetilde{Z}_t^*(iy_n)}\leq\abs{h_t(iy_n)-\widetilde{h}_t(iy_n)}+\abs{\sqrt{\kappa}B_t-\widetilde{\lambda}(t)}.
  \end{equation*}
  Define $\epsilon\coloneqq\sup\limits_{t\in[0,1]}\abs{\sqrt{\kappa}B_t-\widetilde{\lambda}(t)}$, then it follows that
  \begin{equation}
  \label{3.34}
      \abs{Z_t(iy_n)-\widetilde{Z}_t^*(iy_n)}\leq\abs{h_t(iy_n)-\widetilde{h}_t(iy_n)}+\epsilon.
  \end{equation}
  To achieve this goal, we further define $H(t)\coloneqq h_t(iy_n)-\widetilde{h}_t(iy_n)$. And we will first estimate $\abs{H(t)}$. Differentiate $H(t)$ \textit{w.r.t.} $t\in[0,1]$ and use $(\ref{2.2})$ and $(\ref{3.32})$ to obtain
  \begin{equation*}
      \frac{d}{dt}H(t)-H(t)\zeta(t)=(\sqrt{\kappa}\widehat{B}_t-\widehat{\lambda}_t)\zeta(t),
  \end{equation*}
  where we define $\zeta(t)\coloneqq\big(h_t(iy_n)-\sqrt{\kappa}\widehat{B}_t\big)^{-1}\vdot\big(\widetilde{h}_t(iy_n)-\widehat{\lambda}_t\big)^{-1}$. Notice that the derivative of $H(t)$ \textit{w.r.t.} time t is defined except only for finitely many points because the driving force $\widetilde{\lambda}(t)$ is piecewise continuous. Integrating the above differential form and choose $u(t)\coloneqq e^{-\int_0^t\zeta(s)ds}$, we find
  \begin{equation*}
      H(t)=u(t)^{-1}\bigg[H(0)-\int_0^t(\sqrt{\kappa}\widehat{B}_s-\widehat{\lambda}_s)u(s)\zeta(s)\,ds\bigg].
  \end{equation*}
  Since $H(0)=0$, we obtain
  \begin{equation*}
      \abs{H(t)}\leq\int_0^t\abs{\sqrt{\kappa}\widehat{B}_s-\widehat{\lambda}_s}e^{\int_s^t\Re\zeta(r)\,dr}\abs{\zeta(s)}\,ds.
  \end{equation*}
  Then, it is immediate that
  \begin{equation*}\begin{aligned}
      \abs{h_t(iy_n)-\widetilde{h}_t(iy_n)}\leq\epsilon\vdot\int_0^te^{\int_s^t\Re\zeta(r)\,dr}\abs{\zeta(s)}\,ds\leq\epsilon\vdot\bigg(e^{\int_s^t\abs{\zeta(r)}\,dr}-1\bigg),
  \end{aligned}\end{equation*}
  where the last inequality is due to (\cite{Cts}, $Lem.\;2.3$) and (\cite{Cts}, $Eqn.\;2.12$). Now turning our attention to $(\ref{3.34})$, we have 
  \begin{equation*}
     \abs{Z_t(iy_n)-\widetilde{Z}_t^*(iy_n)}\leq\abs{h_t(iy_n)-\widetilde{h}_t(iy_n)}+\epsilon\leq\epsilon\vdot e^{\int_s^t\Re\zeta(r)\,dr}.
 \end{equation*}
 Furthermore, (\cite{Cts}, $Eqn.\;2.12$) tells us that $\int_0^t\abs{\zeta(s)}ds\leq\log(\sqrt{4+y_n^2}/y_n)$. Consequently, we know that
 \begin{equation}
 \label{3.40}
     \abs{Z_t(iy_n)-\widetilde{Z}_t^*(iy_n)}\leq\epsilon\vdot\sqrt{4+y_n^2}/y_n=\epsilon\vdot\sqrt{4n+1}.
 \end{equation}
 Notice that
 \begin{equation*}
     \epsilon=\sup_{t\in[0,1]}\abs{\sqrt{\kappa}B_t-\widetilde{\lambda}(t)}\leq\bigvee\limits_{t_k(t)\in\mathcal{D}_n}\sup_{t\in[0,h_k]}\sqrt{\kappa}\abs{B_t},
 \end{equation*}
where the notation ``$\vee$" indicates taking the maximal value over all $t_k(t)\in\mathcal{D}_n$. By (\cite{Ben}, $Cor.\;2.2$), for the supremum Brownian motion $S_t\coloneqq\sup\limits_{0\leq s\leq t}B_s$ we have
 \begin{equation*}
     \mathbb{P}\bigg(S_t\leq x\bigg)=2\Phi\bigg(\frac{x}{\sqrt{t}}\bigg)-1
 \end{equation*}
 for all $x\geq0$ and where $\frac{d}{dx}\Phi(x)\coloneqq e^{-x^2/2}/\sqrt{2\pi}$ is the density of the standard normal variable. It follows from the reflection principle that
 \begin{equation*}
     \mathbb{P}\bigg(\sup\limits_{0\leq t\leq h_k}\abs{\sqrt{\kappa}B_t}\geq\frac{1}{4n+1}\bigg)=2\mathbb{P}\bigg(S_{h_k}\geq\frac{1}{\sqrt{\kappa}\vdot(4n+1)}\bigg)\leq2\sqrt{\frac{2}{\pi}}e^{-\frac{(4n+1)^{-2}}{2h_k\vdot\kappa}}
 \end{equation*}
 if we restrict  $h_k\leq n^{-1}\wedge(4n+1)^{-3}$ for all $t_k(t)\in\mathcal{D}_n$. Then
  \begin{equation*}
     \bigg\{\epsilon>\frac{1}{4n+1}\bigg\}\subset\bigcup\limits_{t_k(t)\in\mathcal{D}_n}\bigg\{\sup\limits_{0\leq t\leq\frac{h_k}{2}}\abs{\sqrt{\kappa}B_t}>\frac{1}{4n+1}\bigg\},
 \end{equation*}
 and we see
 \begin{equation*}\begin{aligned}
     \mathbb{P}\bigg(\epsilon>\frac{1}{4n+1}\bigg)\leq\sum\limits_{t_k(t)\in\mathcal{D}_n}\mathbb{P}\bigg(\sup\limits_{0\leq t\leq\frac{h_k}{2}}\abs{\sqrt{\kappa}B_t}>\frac{1}{4n+1}\bigg)
     \leq2(4n+1)^3\vdot e^{-(4n+1)/2\kappa}.
  \end{aligned}\end{equation*}
 Conditioned on the event $\Omega\backslash\{\epsilon>(4n+1)^{-1}\}$ and following $(\ref{3.40})$, we have
 \begin{equation*}
     \sup\limits_{t\in[0,1]}\abs{Z_t(iy_n)-\widetilde{Z}_t^*(iy_n)}\leq\frac{1}{\sqrt{4n+1}}.
 \end{equation*}
 Hence, by the strict inclusion of events in probability space, we have our following estimate to the perturbation term 
 \begin{equation*}
     \mathbb{P}\bigg(\sup\limits_{t\in[0,1]}\abs{Z_t(iy_n)-\widetilde{Z}_t^*(iy_n)}\leq\frac{1}{\sqrt{4n+1}}\bigg)\geq1-2(4n+1)^3\vdot e^{-(4n+1)/2\kappa}.
 \end{equation*}
 We further observe that the splitting scheme $\widetilde{Z}_t(iy_n)$ coincides with the trajectory $\widetilde{Z}_t^*(iy_n)$ at the times $t_k(t)\in\mathcal{D}_n$, by virtue of $Lem.\;3.13$. Hence we have the desired result
 \begin{equation*}
     \mathbb{P}\big(E_{n,2}\big)\geq1-2(4n+1)^3\vdot e^{-(4n+1)/2\kappa}.
 \end{equation*}
\end{proof}
\begin{remark}
The perturbation in $Prop.\;3.14$ in our splitting scheme is estimated via a probabilistic argument using (\cite{Cts}, $Lem.\;2.2$). Notice that the time-reversed Loewner map $h_t(z)$ is different from the inverse of the forward map $g_t^{-1}(z)$, even though we do have the equality at the final moment $T$ by $h_{T=1}(z)=g_{T=1}^{-1}(z)$. In $Sec.\;4.$ we are going to briefly discuss the linear interpolation of driving force. To prove its convergence, we will need (\cite{Cts}, $Lem.\;2.2$) again under a different context.
\end{remark}
Hence we have estimated the sup-norm on $[0,1]$ of the second term in $(\ref{3.2a})$. Next, we will estimate the $\norm{\vdot}_{[0,1],\infty}$ norm of the third term. Following $(\ref{2.8})$ with $L_0(z)=-2/z\;\text{and}\;L_1(z)=\sqrt{\kappa}$, we calculate with $t_k(t)\leq s<t_{k+1}(t)$ that
\[
\widetilde{Z}^{(0)}_s=\exp\bigg(\frac{1}{2}(s-t_k(t))L_0\bigg)\widetilde{Z}_{t_k(t)}=\sqrt{\widetilde{Z}^2_{t_k(t)}-2(s-t_k(t))},
\]
as well as 
\[
\widetilde{Z}^{(1)}_s=\exp\bigg(B_{t_k(t),s}L_1\bigg)\widetilde{Z}^{(0)}_{t_{k+1}}=\sqrt{\widetilde{Z}^2_{t_k(t)}-2h_k}+B_{t_k(t),s},
\]
and 
\[
\widetilde{Z}^{(2)}_s=\exp\bigg(\frac{1}{2}(s-t_k(t))L_0\bigg)\widetilde{Z}^{(1)}_{t_{k+1}}=\sqrt{\big(\widetilde{Z}^{(1)}_{t_{k+1}(t)}\big)^2-2(s-t_k(t))}, 
\]
which could be written into the form via solving $(\ref{2.6})$ so that
\begin{equation*}\begin{aligned}
    &\widetilde{Z}_t(iy_n)-\widetilde{Z}_{t_k(t)}(iy_n)=\frac{1}{2}\int_{t_k(t)}^tL_0(\widetilde{Z}_s^{(2)})\,ds+\int_{t_k(t)}^tL_1(\widetilde{Z}_s^{(1)})\,dB_s+\frac{1}{2}\int_{t_k(t)}^tL_0(\widetilde{Z}_s^{(0)})\,ds,\\
    &\quad=\int_{t_k}^t\sqrt{\kappa}\,dB_s-\int_{t_k}^t\frac{1}{\sqrt{(\widetilde{Z}_{t_{k}})^2-2(s-t_k)}}\,ds-\int_{t_k}^t\frac{1}{\sqrt{\big(\widetilde{Z}_{t_{k+1}}^{(1)}\big)^2-2(s-t_k)}}\,ds,
\end{aligned}\end{equation*}
where $t_k,t_{k+1}$ abbreviates $t_{k}(t),t_{k+1}(t)$ in the above expression, respectively. Solving these integrals, we have
\begin{equation*}\begin{aligned}
\label{3.48b}
    &\widetilde{Z}_t(iy_n)-\widetilde{Z}_{t_k(t)}(iy_n)=\sqrt{\kappa}B_{t_k(t),t}-\frac{2(t-t_k(t))}{\sqrt{\widetilde{Z}_{t_k(t)}^2-2(t-t_k(t))}+\widetilde{Z}_{t_k(t)}}\\
    &\qquad-\frac{2(t-t_k)}{\sqrt{(\sqrt{\widetilde{Z}_{t_k}^2-2h_k}+\sqrt{\kappa}B_{t_k,t_{k+1}})^2-2(t-t_k)}+\sqrt{\widetilde{Z}_{t_k}^2-2h_k}+\sqrt{\kappa}B_{t_k,t_{k+1}}}.
\end{aligned}\end{equation*}
Notice that $(\ref{3.48b})$ provides an exact form of the approximated process $\widetilde{Z}_t$, which leads us to the following result.
\begin{proposition} 
Let $\kappa \neq 8$. And consider the following event
\begin{equation*}
    E_{n,3}\coloneqq\bigg\{\norm{\widetilde{Z}_t(iy_n)-\widetilde{Z}_{t_k(t)}(iy_n)}_{[0,1],\infty}\leq\frac{2}{n^{1/2}}+\frac{1}{n^{1/4}}\bigg\}.
\end{equation*}
Then as long as $||\mathcal{D}_n||\leq n^{-1}\wedge(4n+1)^{-3}$, we have
\begin{equation*}
    \mathbb{P}\big(E_{n,3}\big)\geq1-\frac{1}{n}-2e^{-\sqrt{n}/2\kappa}.
\end{equation*}
\end{proposition}
\begin{proof}
  By (\cite{Foster}, $Sec.\;6.1$), we have two general results $\Im(z)\leq\Im(\sqrt{z^2-c})$ and $\Im(z)=\Im(z+c)$ for all $z\in\mathbb{H}$ and $c\in\mathbb{R}$. Applying this two results to $(\ref{3.48b})$, we have the following inequality
  \begin{equation}\begin{aligned}\label{3.52}
      \abs{\widetilde{Z}_t(iy_n)-\widetilde{Z}_{t_k(t)}(iy_n)}\leq\abs{\sqrt{\kappa}B_{t_k(t),t}}+\frac{h_k}{\Im\widetilde{Z}_{t_k(t)}(iy_n)}+\frac{h_k}{\Im\widetilde{Z}_{t_k(t)}(iy_n)}\leq\abs{\sqrt{\kappa}B_{t_k(t),t}}+\frac{2h_k}{y_n},
  \end{aligned}\end{equation}
  where the last inequality follows from (\cite{Antti}, $Lem.\;4.9$), in which it is shown that the map $t\mapsto\Im\widetilde{Z}_t(iy)$ is strictly increasing. Using (\cite{Ben}, $Cor.\;2.2$), we obtain
\begin{equation*}
    \mathbb{P}\bigg(\sup_{0\leq t\leq h_k}\abs{\sqrt{\kappa}B_t}\geq\frac{1}{n^{1/4}}\bigg)=2\mathbb{P}\bigg(S_{h_k}\geq \frac{1}{\sqrt{\kappa}\vdot n^{1/4}}\bigg)\leq 2\sqrt{\frac{2}{\pi}}e^{-\frac{n^{-1/2}}{2h_k\vdot\kappa}},
\end{equation*}
by reflection principle. Note that we have restricted $h_k\leq n^{-1}\wedge(4n+1)^{-3}$ for all $t_k(t)\in\mathcal{D}_n$. In this regard
  \begin{equation*}
      \mathbb{P}\bigg(\sup\limits_{t\in[0,1]}\abs{\sqrt{\kappa}B_{t_k(t),t}}\geq\frac{1}{n^{1/4}}\bigg)\leq2e^{-\sqrt{n}/2\kappa},
  \end{equation*}
  and hence
  \begin{equation*}
      \mathbb{P}\big(E_{n,3}\big)\geq1-\frac{1}{n}-2e^{-\sqrt{n}/2\kappa}.
  \end{equation*}
\end{proof}
At this point, we have evaluated the sup-norm $w.r.t.$ all the three terms in $(\ref{3.2a})$. Hence, it is time that we come to prove our main result.
\begin{proof}{{\rm (\textit{of Thm. 3.4})}} Define $E_{n,4}\coloneqq E^{**}_{n,1}$. On the event $E_{n,1}\cap E_{n,2}\cap E_{n,3}\cap E_{n,4}\subseteq\Omega$. From $Prop.\;3.10,\;Prop.\;3.11,\;Prop.\;3.14,\;\text{and}\;Prop.\;3.16$ we observe
\begin{equation}\begin{aligned}\label{four terms}
    &\norm{\eta(t)-Z_t(iy_n)}_{[0,1],\infty}\leq\frac{1}{n^{(1-\epsilon_0)/4}},\\
    &\norm{Z_t(iy_n)-Z_{t_k(t)}(iy_n)}_{[0,1],\infty}\leq\frac{2\phi(\sqrt{n})}{(1-\beta_1)n^{(1-\beta_1)/2}}+\frac{2}{n^{(1-\epsilon_0)/4}},\\
    &\norm{Z_{t_k(t)}(iy_n)-\widetilde{Z}_{t_k(t)}(iy_n)}_{[0,1],\infty}\leq\frac{1}{(4n+1)^{1/2}},\\
    &\norm{\widetilde{Z}_{t_k(t)}(iy_n)-\widetilde{Z}_t(iy_n)}_{[0,1],\infty}\leq\frac{2}{n^{1/2}}+\frac{1}{n^{1/4}},
\end{aligned}\end{equation}
 given $||\mathcal{D}_n||\leq n^{-1}\wedge(4n+1)^{-3}$ and $\beta_1 \in (0,1)$. If we define
  \begin{equation}\begin{aligned}
  \label{result1}
      \varphi_1(n)\coloneqq\frac{2\phi(\sqrt{n})}{(1-\beta_1)n^{(1-\beta_1)/2}}+\frac{3}{n^{(1-\epsilon_0)/4}}+\frac{1}{(4n+1)^{1/2}}+\frac{2}{n^{1/2}}+\frac{1}{n^{1/4}}
      \end{aligned}\end{equation}
      and
      \begin{equation}\begin{aligned}
      \label{result2}
      \varphi_2(n)\coloneqq\frac{1}{n}+\frac{c_2}{n^2}+\frac{c_4}{n^{c_3/2}}+2(
      4n+1)^3\vdot e^{-(4n+1)/2\kappa}+2e^{-\sqrt{n}/2\kappa}+3\epsilon_n,
  \end{aligned}\end{equation}
 then $\varphi_1(n)\to0$ and $\varphi_2(n)\to0$ as $n\to\infty$. Therefore the following inequality proves our result
 \begin{equation*}
     \mathbb{P}\bigg(\norm{\eta(t)-\widetilde{Z}_t(iy_n)}_{[0,1],\infty}\leq\varphi_1(n)\bigg)\geq\mathbb{P}\bigg(E_{n,1}\cap \cdots\cap E_{n,4}\bigg)\geq1-\varphi_2(n).
 \end{equation*}
\end{proof}
\begin{corollary}
For almost all $\omega\in\Omega$, we have that
\begin{equation*}
    \norm{\eta(t)-\widetilde{Z}_t(iy_n)}_{[0,1],\infty}\to0,\;\text{with}\;n\to\infty.
\end{equation*}
\end{corollary}
\begin{proof}
This corollary immediately follows from the uniform convergence in probability in $Thm.\;3.4$.
\end{proof}

\begin{remark}
{A classic method to improve the overall approximation effect is via refining the partition used, called \textit{tolerance}.} From a practical view-point, for a fixed {algorithm} step $n \in \mathbb{N}$, in order to achieve better precision one can choose a constant $\tau>0$, called tolerance, to ensure that
\begin{equation*}
\abs{\widetilde{Z}_{t_{k+1}}-\widetilde{Z}_{t_{k}}}\leq\tau
\end{equation*}
for each $k$. To achieve this, we start by computing $\widetilde{Z}_t$ along a prior uniform partition until $\abs{\widetilde{Z}_{t_{k+1}}-\widetilde{Z}_{t_{k}}}>\tau$. If this event occurs, we reduce the step size $h_k$ of the $SLE_\kappa$ discretization, that is we insert the mid-point of this interval $[t_k(t),t_{k+1}]$ into the partition.\par
{It is worthy to point out that this resulting refined new partition is \textit{random}, so that further adaptions should be made to incorporate this idea of tolerance into the splitting algorithm  in Section \ref{sec:chapter 1}, along with verifying its convergence.}\par
 Notice that the choice of a refined partition actually depends on $\omega\in\Omega$ because the evolution $(\widetilde{Z}_t)_{0\leq t\leq1}$ contains Brownian motion.
\end{remark}

{
\section{A version of $L^p$-norm convergence}
Following (\ref{3.2a}), we view the convergence of the Ninomiya--Victoir splitting algorithm to $SLE_\kappa$ from a different metric perspective. The natural regularity of the $SLE_\kappa$ dynamics allows us to reveal the convergence of its splitting approximation via $L^p$-norm topology with $p\geq2$. The result is w.r.t. a somewhat \textit{twisted} $p^{th}$-moment, whose precise meaning will be explained in the following.

\begin{theorem}
Let $\eta(t)$ be the backward $SLE_\kappa$ trace for $\kappa \neq 8$.
There exist two non-increasing functions $\psi_i:\mathbb{N}\to\mathbb{R}_+$ so that $\lim_{n\to\infty}\psi_i(n)=0^+\;\text{with}\;i=1,2$. If $p\geq2$ and the mesh $||\mathcal{D}_n||\to0^+$ with $n\to\infty$, at a rate  $||\mathcal{D}_n||=o(n^{-3})$, then
\begin{equation*}
    \mathbb{E}\bigg[\int_0^1\abs{\eta(t)-\widetilde{Z}_t(iy_n)}^p\,dt\vdot I_{\mathcal{A}_n}\bigg]\leq\psi_1(n),\;\;\text{with}\;\;\mathbb{P}(\mathcal{A}_n)\geq1-\psi_2(n).
\end{equation*}
\end{theorem}
\begin{remark}
    The phrase \textit{a version} refers to the fact that we are conditioning on an event $\mathcal{A}_n$ whose size approximates the probability space $\Omega$ with velocity $\psi_2(n)$. The exact form of decreasing functions $\psi_1$ and $\psi_2$ can be found later in (\ref{result1}) and (\ref{result2}). Therefore, we have verified this version of $L^p$-norm convergence of the Ninomiya--Victoir splitting algorithm in the context of $SLE_\kappa$.
\end{remark}
To verify Theorem $4.1$, we need to state and show a few propositions.

\begin{proposition}
With $||\mathcal{D}_n||\leq n^{-1}\wedge(4n+1)^{-3}$, there exist $\mathcal{A}_n^{(1)}\subseteq\Omega$ and decreasing functions $\widetilde{\psi}_{(1)},\widehat{\psi}_{(1)}:\mathbb{N}\to\mathbb{R}_+$ such that $\widetilde{\psi}_{(1)}(n)\to0$ and $\widehat{\psi}_{(1)}(n)\to0$ as $n\to\infty$, and 
\begin{equation*}
    \mathbb{E}\bigg[\displaystyle\int_0^1\abs{\eta (t)-Z_t(iy_n)}^{p}\,dt\vdot I_{\mathcal{A}_n^{(1)}}\bigg]\leq\widetilde{\psi}_{(1)}(n)\;\;\text{with}\;\;\mathbb{P}(\mathcal{A}_n^{(1)})\geq1-\widehat{\psi}_{(1)}(n).
\end{equation*}
\end{proposition}
\begin{proof}
By (\cite{Foster}, $Lem.\;6.7$), there exists $\epsilon_0\in(0,1)$ such that \textit{almost surely}
\begin{equation*}
    \sup\limits_{t\in[0,1]}\abs{\eta(t)-Z_t(iy_n)}\leq C_1(\omega)\vdot y_n^{1-\epsilon_0},
\end{equation*}
with $\mathbb{P}$-a.s. finite $C_1(\omega)$. Take $\mathcal{A}_n^{(1)}\coloneqq\{C_1(\omega)< n^{(1-\epsilon_0)/4}\}$. Then $\widehat{\psi}_{(1)}(n)\coloneqq\mathbb{P}(C_1(\omega)\geq n^{(1-\epsilon_0)/4})\to0^+$. It is clear then
\begin{equation*}\begin{aligned}
    &\mathbb{E}\bigg[\displaystyle\int_0^1\abs{\eta (t)-Z_t(iy_n)}^{p}\,dt\vdot I_{\mathcal{A}_n^{(1)}}\bigg]\leq
    \mathbb{E}\bigg[\displaystyle\norm{\eta(t)-Z_t(iy_n)}^{p}_{[0,1],\infty}\vdot I_{\mathcal{A}_n^{(1)}}\bigg]\\
    &\qquad \leq n^{(1-\epsilon_0)p/4}\vdot\frac{1}{n^{(1-\epsilon_0)p/2}}\coloneqq\widetilde{\psi}_{(1)}(n)\stackrel{n\to\infty}{\longrightarrow}0^+.
\end{aligned}\end{equation*}
And the assertion is verified.
\end{proof}

\begin{proposition}
With $||\mathcal{D}_n||\leq n^{-1}\wedge(4n+1)^{-3}$, there exist $\mathcal{A}_n^{(2)}\subseteq\Omega$ and decreasing functions $\widetilde{\psi}_{(2)},\widehat{\psi}_{(2)}:\mathbb{N}\to\mathbb{R}_+$ such that $\widetilde{\psi}_{(2)}(n)\to0$ and $\widehat{\psi}_{(2)}(n)\to0$ as $n\to\infty$, and  
\begin{equation*}
    \mathbb{E}\bigg[\displaystyle\int_0^1\abs{Z_t(iy_n)-Z_{t_k(t)}(iy_n)}^{p}\,dt \vdot I_{\mathcal{A}_n^{(2)}} \bigg]\leq\widetilde{\psi}_{(2)}(n)\;\;\text{with}\;\;\mathbb{P}(\mathcal{A}_n^{(2)})\geq1-\widehat{\psi}_{(2)}(n).
\end{equation*}
\end{proposition}
\begin{proof}
(\cite{Fredrik}, $Prop.\;3.8$) and (\cite{Fredrik}, $Prop.\;4.3$) imply that there exists $\beta_2\in(0,1)$ such that 
\begin{equation*}
    \sup\limits_{t\in[0,1]}\abs{\eta(t)-\eta\big(t_k(t)\big)}\leq C_2(\omega)\vdot\frac{1}{n^{(1-\beta_2)/2}}.
\end{equation*}
By $Prop.\;4.4$, it follows that 
\begin{equation}\label{4.9}
    \sup\limits_{t\in[0,1]}\abs{Z_t(iy_n)-Z_{t_k(t)}(iy_n)}\leq\frac{2C_1(\omega)}{n^{(1-\epsilon_0)/2}}+\frac{C_2(\omega)}{n^{(1-\beta_2)/2}}.
\end{equation}
Take $\mathcal{A}_n^{(2)}\coloneqq\mathcal{A}_n^{(1)}\cap
\{C_2(\omega)< n^{(1-\beta_2)/4}\}$, then $\widehat{\psi}_{(2)}(n)\coloneqq \widehat{\psi}_{(1)}(n) + \mathbb{P}(C_2(\omega)\geq n^{(1-\beta_2)/4})\to0^+$.
It is clear then
\begin{equation*}\begin{aligned}
    &\mathbb{E}\bigg[\displaystyle\int_0^1\abs{Z_t(iy_n)-Z_{t_k(t)}(iy_n)}^{p}\,dt\vdot I_{\mathcal{A}_n^{(2)}} \bigg] \leq \mathbb{E}\bigg[\displaystyle\norm{Z_t(iy_n)-Z_{t_k(t)}(iy_n)}_{[0,1],\infty}^{p}\vdot I_{\mathcal{A}_n^{(2)}} \bigg]\\ &\qquad\leq n^{(1-\epsilon_0)p/4}\vdot\frac{2^{2p-1}}{n^{(1-\epsilon_0)p/2}} + n^{(1-\beta_2)p/4}\vdot \frac{2^{p-1}}{n^{(1-\beta_2)p/2}}\coloneqq\widetilde{\psi}_{(2)}(n)\stackrel{n\to\infty}{\longrightarrow}0^+.
\end{aligned}\end{equation*}
And the assertion is verified.
\end{proof}

\begin{proposition}
With $||\mathcal{D}_n||\leq n^{-1}\wedge(4n+1)^{-3}$, there exist $\mathcal{A}_n^{(3)}\subseteq\Omega$ and decreasing functions $\widetilde{\psi}_{(3)},\widehat{\psi}_{(3)}:\mathbb{N}\to\mathbb{R}_+$ such that $\widetilde{\psi}_{(3)}(n)\to0$ and $\widehat{\psi}_{(3)}(n)\to0$ as $n\to\infty$, and 
\begin{equation*}
    \mathbb{E}\bigg[\int_0^1\abs{Z_{t_k(t)}(iy_n)-\widetilde{Z}_{t_k(t)}(iy_n)}^p\,dt\vdot I_{\mathcal{A}_n^{(3)}}\bigg] \leq \widetilde{\psi}_{(3)}(n)\;\;\text{with}\;\;\mathbb{P}(\mathcal{A}_n^{(3)})\geq1-\widehat{\psi}_{(3)}(n).
\end{equation*}
\end{proposition}
\begin{proof}
Invoking Proposition 3.15 and taking $\mathcal{A}_n^{(3)}\coloneqq E_{n,2}$, we have $\widehat{\psi}_{(3)}(n)\coloneqq2e^{-(4n+1)/\kappa}\to0^+$. Moreover,
\begin{equation*}\begin{aligned}
    &\mathbb{E}\bigg[\int_0^1\abs{Z_{t_k(t)}(iy_n)-\widetilde{Z}_{t_k(t)}(iy_n)}^p\,dt\vdot I_{\mathcal{A}_n^{(3)}}\bigg]\leq \mathbb{E}\bigg[\sup\limits_{t\in[0,1]}\abs{Z_{t_k(t)}(iy_n)-\widetilde{Z}_{t_k(t)}(iy_n)}^p\vdot I_{\mathcal{A}_n^{(3)}}\bigg]\\
    &\qquad\leq \frac{1}{(4n+1)^{p/2}}\coloneqq\widetilde{\psi}_{(3)}(n)\stackrel{n\to\infty}{\longrightarrow}0^+,
\end{aligned}\end{equation*}
verifying the assertion.
\end{proof}

To give an estimate to the last term in (\ref{four terms}), we quote the known interpolation inequality from (\cite{Folland}, $Sec.\;6.5$) for Lebesgue spaces and another inequality $w.r.t.$ supremum Brownian motion.
\begin{lemma}
For all $1<p<r<q$, suppose $f\in L^p\cap L^q$. Then $f\in L^r$ with
\begin{equation*}
    \norm{f}_r\leq\big(\norm{f}_p\big)^{(1/r-1/q)/(1/p-1/q)}\big(\norm{f}_q\big)^{(1/r-1/q)/(1/p-1/q)}.
\end{equation*}
\end{lemma}
\begin{lemma}
For all $m\in\mathbb{N}$, the standard Brownian motion enjoys
\begin{equation*}
    \mathbb{E}\bigg[\sup\limits_{s\in[0,t]}\abs{B_s}^{2m}\bigg]=\pi^{-\frac{1}{2}}\Gamma\bigg(\frac{1}{2}+m\bigg)2^m\vdot t^m.
\end{equation*}
\end{lemma}
\begin{proof}
The proof follows if we revisit the supremum Brownian motion $S_t=\sup_{0\leq s\leq t}B_s$ from (\cite{Ben}, $Cor.\;2.2$). by computing all even order moments.
\end{proof}
\begin{proposition}
With $||\mathcal{D}_n||\leq n^{-1}\wedge(4n+1)^{-3}$, there exists a decreasing function $\widetilde{\psi}_{(4)}:\mathbb{N}\to\mathbb{R}_+$ such that $\widetilde{\psi}_{(4)}(n)\to0^+$ with $n\to\infty$, and
\begin{equation*}
    \mathbb{E}\bigg[\int_0^1\abs{\widetilde{Z}_{t_k(t)}(iy_n)-\widetilde{Z}_t(iy_n)}^p\,dt\bigg] \leq\widetilde{\psi}_{(4)}(n).
\end{equation*}
\end{proposition}
\begin{proof}
By Proposition 3.17 and (\ref{3.52}), we know \textit{almost surely} that
\begin{equation*}
    \abs{\widetilde{Z}_{t_k(t)}(iy_n)-\widetilde{Z}_t(iy_n)}^p\leq2\abs{\sqrt{\kappa}B_{t_k(t),t}}^p+\frac{2^{p+1}}{n^{p/2}}.
\end{equation*}
Since $p\geq2$, there exists $\{m,m+1\}\subset\mathbb{N}$ so that $2m\leq p<2(m+1)$. Remember that $t-t_k(t)\leq h_k\leq n^{-1}$. Then, by $Lem.\;4.8$
\begin{equation*}\begin{aligned}
    &\mathbb{E}\bigg[\sup\limits_{t\in[0,1]}\abs{\sqrt{\kappa}B_{t_k(t),t}}^{2m}\bigg]\leq\frac{2^m\kappa^m\Gamma\big(\frac{1}{2}+m\big)}{\pi^{\frac{1}{2}} n^m},\\
    &\mathbb{E}\bigg[\sup\limits_{t\in[0,1]}\abs{\sqrt{\kappa}B_{t_k(t),t}}^{2m+2}\bigg]\leq\frac{2^{m+1}\kappa^{m+1}\Gamma\big(\frac{3}{2}+m\big)}{\pi^{\frac{1}{2}} n^{m+1}}.
\end{aligned}\end{equation*}
By $Lem.\;4.7$, we use the interpolation inequality in Lebesgue spaces and have
\begin{equation*}\begin{aligned}
    &\mathbb{E}\bigg[\sup\limits_{t\in[0,1]}\abs{\sqrt{\kappa}B_{t_k(t),t}}^{p}\bigg]\\&\;\leq\mathbb{E}\bigg[\sup\limits_{t\in[0,1]}\abs{\sqrt{\kappa}B_{t_k,t}}^{2m}\bigg]^{m+1-\frac{p}{2}}\mathbb{E}\bigg[\sup\limits_{t\in[0,1]}\abs{\sqrt{\kappa}B_{t_k,t}}^{2m+2}\bigg]^{(m+1-\frac{p}{2})\vdot\frac{m}{m+1}}\\&\;\leq\frac{c_7}{n^{m(m+1-\frac{p}{2})}}\coloneqq\widetilde{\psi}_{(4)}(n) \stackrel{n\to\infty}{\longrightarrow}0^+,
\end{aligned}\end{equation*}
where $c_7>0$ is a constant depending only on $p\geq2$.
\end{proof}

\begin{proof}{{\rm (\textit{of Thm. 4.1})}}
    Take $\mathcal{A}_n\coloneqq\mathcal{A}_n^{(1)}\cap\mathcal{A}_n^{(2)}\cap\mathcal{A}_n^{(3)}$. Then $\mathbb{P}(\mathcal{A}_n)\geq1-\psi_2(n)$, where $\psi_2(n)\coloneqq\widehat{\psi}_{(1)}(n) + \widehat{\psi}_{(2)}(n) + \widehat{\psi}_{(3)}(n)\to0^+$. Moreover,
    \begin{equation*}\begin{aligned}                        
        &\quad\mathbb{E}\bigg[\int_0^1\abs{\eta(t)-\widetilde{Z}_t(iy_n)}^p\,dt\vdot I_{\mathcal{A}_n}\bigg]^{1/p} \leq \mathbb{E}\bigg[\displaystyle\int_0^1\abs{\eta (t)-Z_t(iy_n)}^{p}\,dt\vdot I_{\mathcal{A}_n^{(1)}}\bigg]^{1/p}\\
        &+ \mathbb{E}\bigg[\displaystyle\int_0^1\abs{Z_t(iy_n)-Z_{t_k(t)}(iy_n)}^{p}\,dt \vdot I_{\mathcal{A}_n^{(2)}} \bigg]^{1/p} + \mathbb{E}\bigg[\int_0^1\abs{Z_{t_k(t)}(iy_n)-\widetilde{Z}_{t_k(t)}(iy_n)}^p\,dt\vdot I_{\mathcal{A}_n^{(3)}}\bigg]^{1/p}\\
        &+\mathbb{E}\bigg[\int_0^1\abs{\widetilde{Z}_{t_k(t)}(iy_n)-\widetilde{Z}_t(iy_n)}^p\,dt\bigg]^{1/p}\leq \widetilde{\psi}_{(1)}(n)^{1/p} + \widetilde{\psi}_{(2)}(n)^{1/p} + \widetilde{\psi}_{(3)}(n)^{1/p} + \widetilde{\psi}_{(4)}(n)^{1/p}\\
        &\coloneqq \psi_1(n)^{1/p}\stackrel{n\to\infty}{\longrightarrow}0^+.
    \end{aligned}\end{equation*}
    And the assertion is verified.
\end{proof}

}

\section{Linear interpolation of driving forces}
\label{sec:chapter 3}
In this last section we study the approximation obtained from the piecewise linear interpolation of the Brownian driving force. This method is independent from the study of the Ninomiya-Victoir splitting {algorithm} in the first part of this paper. Let us briefly discuss this idea to simulate driving force $\lambda(t)$ and its corresponding hull $K_t\;\text{for all}\;t\in[0,T]$. The interpolation algorithm is based on the following observations. Fix $s>0$, let $(\widetilde{g}_t)_{0\leq t\leq T}$ be the Loewner chain driven by the continuous driving force $\widetilde{\lambda}(t)=\lambda(s+t),$ with $0\leq t\leq T-s$. Using the format of the forward Loewner differential equation, we have 
\begin{equation*}
\partial_tg_{s+t}\circ g_s^{-1}(z)=\frac{2}{g_{s+t}\circ g_s^{-1}(z)-\lambda(s+t)}=\frac{2}{g_{s+t}\circ g_s^{-1}(z)-\widetilde{\lambda}(t)} 
\end{equation*}
and $g_{s}\circ g_s^{-1}(z)=z$ with $z\in\mathbb{H}\backslash K_s$. The uniqueness property of Loewner chains implies that $\widetilde{g}_t(z)=g_{s+t}\circ g_s^{-1}(z)$. We denote $\widetilde{K}_t$ to be the hull associated with the Loewner chain $\widetilde{g}_t$. Indeed 
\begin{equation*}
\mathbb{H}\cap g_s(K_{s+t})=\widetilde{K}_t\;\text{ and }\; K_{s+t}=K_s\cup g_s^{-1}(\widetilde{K}_t).   
\end{equation*}Hence, computing $K_s$, $g_s^{-1}$ and $\widetilde{K}_t$ would enable us to compute $\widetilde{K}_{s+t}$.In order to implement the linear-interpolation of the Brownian driving force, we consider the following linear interpolation of a general driving force given by
\begin{equation*}
    \lambda^{n}_{\textit{linear}}(t)\coloneqq n\big(\lambda(t_{k+1})-\lambda(t_k)\big)(t-t_k)+\lambda(t_k)\;\text{on}\;[t_k,t_{k+1}].
\end{equation*}
In \cite{Tran} the case of square-root interpolation of a general driving force has been previously considered with 
\begin{equation*}
    \lambda^{n}_{\textit{square-root}}(t)\coloneqq\sqrt{n}\big(\lambda(t_{k+1})-\lambda(t_k)\big)\sqrt{t-t_k}+\lambda(t_k)\;\text{on}\;[t_k,t_{k+1}]\;\;\:\:
\end{equation*}
in order to approximate the $SLE_\kappa$ trace. 
These methods are based on the fact that for driving force with the form $c\sqrt{t}+d$ for $c, d, \in \mathbb{R}$ and $ct$ for $c \in \mathbb{R}$ the Loewner maps and curves can be explicitly computed (see \cite{Fredrikk}). We refer the readers to \cite{Tran} for more details.
\begin{definition}
As before, $\mathcal{D}_n\coloneqq\{t_0=0,t_1,\ldots,t_n=1\}$ will denote a uniform partition on the compact time interval $[0,1]$.
\end{definition}
\begin{definition}
Given the Brownian sample paths $B_{\vdot}(\omega):[0,1]\to\mathbb{R}$, we write our linear interpolation in the following form
\begin{equation}
\label{4.5}
    \lambda^{n}(t)\coloneqq n\big(\sqrt{\kappa}B_{t_{k+1}}-\sqrt{\kappa}B_{t_k(t)}\big)(t-t_k(t))+\sqrt{\kappa}B_{t_k(t)}\;\text{on}\;[t_k(t),t_{k+1}].
\end{equation}
\end{definition}
\begin{definition}
Following the convention in $Sec.\;2.$ and $Sec.\;3.$ we use the notation $\gamma:[0,1]\to {\overline{\mathbb{H}}}$ for the forward Loewner curve generated by the forward Loewner chain $(g_t)_{t\in[0,1]}$ with driving force $\sqrt{\kappa}B_t$. We further let $\gamma^n:[0,1]\to{\overline{\mathbb{H}}}$ to denote the forward Loewner curve generated by the forward Loewner chain $(g^n_t)_{t\in[0,1]}$, driven by the piecewise-linear force $\lambda^n(t)$.
\end{definition}
\begin{remark}
   Notice that in this section we will interpolate forward Loewner chain and hence simulate the forward Loewner curve $\gamma(t)$, whereas we have simulated the backward Loewner curve $\eta(t)$ via Ninomiya--Victoir in $Sec.\;2.$
\end{remark}
\begin{definition}
With $(g^n_t)_{t\in[0,1]}$ the Loewner chain corresponding to $\lambda^n(t)$, let $f^n_t:\mathbb{H}\to\mathbb{H}\backslash\gamma^n([0,t])$ be the inverse map of $g^n_t(z)$ and denote by $\widehat{f}^n_t(z)\coloneqq f^n_t\big(z+\lambda^n(t)\big)$. Choose $G^n_k\coloneqq(\widehat{f}^n_{t_k(t)})^{-1}\circ\widehat{f}^n_{t_{k+1}}$. Then
\begin{equation*}
    \widehat{f}^n_{t_k(t)}=G^n_0\circ G^n_1\circ\cdots\circ G^n_{k-1}.
\end{equation*}
\end{definition}
\begin{definition}
Choose $\gamma^n_t(s)\coloneqq g^n_t\big(\gamma^n(t+s)\big)$ with $s\in[0,1-t]$, for all $t\in[0,1]$.
\end{definition}
\begin{lemma}
Consider the event $F_{n,1}\coloneqq E_{n,1}^\prime$ and $F_{n,2}\coloneqq E_{n,1}^{\prime\prime}$ as in $(\ref{3.3})$ and in $(\ref{3.5})$. Then we have
\begin{equation}
\label{4.7}
    \mathbb{P}\big(F_{n,1}\big)\geq1-\frac{c_2}{n^2}\;\;\;\text{and}\;\;\;\mathbb{P}\big(F_{n,2}\big)\geq1-\frac{c_4}{n^{c_3/2}},
\end{equation}
where $c_2,c_3,c_4>0$ are constants depending only on $\kappa\neq8$.
\end{lemma}
\begin{theorem}
Let $\gamma$ be the $SLE_{\kappa}$ trace and let $\gamma^n$ be the trace obtained by the linear interpolation of the Brownian driving force.
There exist $c_6,c_7>0$ and $\beta_1\in(0,1)$ depending only on $\kappa\neq8$ such that if we consider the event 
\begin{equation*}
    F_{n}\coloneqq\bigg\{\norm{\gamma-\gamma^n}_{[0,1],\infty}\leq\frac{c_6(\log n)^{c_7}}{n^{(1-\sqrt{(1+\beta_1)/2})/2}}\bigg\}.
\end{equation*}
Then we have $\mathbb{P}\big(F_n\big)\geq 1-c_2\vdot n^{-2}-c_3\vdot n^{-c_4/2}$.
\end{theorem}
This theorem is our main result in the context of linearly interpolating Brownian driving force. We will not give a detailed proof here because the proof is similar to the known result of square-root interpolation obtained in (\cite{Tran}, $Sec.\;2.$). Instead, we outline the ideas to estimate the uniform convergence in probability of the linear interpolation method.\par
On the event $F_{n,1}\cap F_{n,1}$, we want to give a uniform bound to $\abs{\gamma(t)-\gamma^n(t)}$ with $t\in[0,1]$. In fact, for all $t_k(t)\in\mathcal{D}_n$ we write
\begin{equation}\begin{aligned}
\label{4.9}
    &\abs{\gamma(r+t_k)-\gamma^n(r+t_k)}\leq\abs{\gamma(r+t_k)-\gamma(s+t_k)}+\abs{\widehat{f}_{t_k}(z)-\widehat{f}^n_{t_k}(w)}\\
    &\;\;\;\;\;\;\;\;\leq\abs{\gamma(r+t_k)-\gamma(s+t_k)}+\abs{\widehat{f}_{t_k}(z)-\widehat{f}_{t_k}(w)}+\abs{\widehat{f}_{t_k}(w)-\widehat{f}_{t_k}^n(w)},
\end{aligned}\end{equation}
where $t_k$ abbreviates $t_k(t)$ and $w\coloneqq\gamma^n_k(r)$, $r$ is arbitrarily fixed in $[\frac{1}{n},\frac{2}{n}]$ and $z\coloneqq\gamma_k(s)$ is chosen to be the highest point in the arc $\gamma_k([0,\frac{2}{n}])$. The first term in $Ineq.\;(\ref{4.9})$ is bounded by the uniform continuity of $\gamma(t)$ on the event $F_{n,1}$.\par

The estimate of the second term in $(\ref{4.9})$ is comparing the images of nearby points in $\mathbb{H}$ very close to the real and the imaginary axis, under a conformal map. To proceed our discussion, we introduce, for any subpower function $\phi:\mathbb{R}_+\to\mathbb{R}_+$, constant $c>0$, and integer $n\in\mathbb{N}_+$ that
\begin{equation*}
    A_{n,c,\phi}\coloneqq\bigg\{x+iy\in\mathbb{H};\;\abs{x}\leq\frac{\phi(n)}{\sqrt{n}}\;\text{and}\;\frac{1}{\sqrt{n}\phi(n)}\leq y\leq\frac{c}{\sqrt{n}}\bigg\},
\end{equation*}
which is a box near the origin in  the upper half-plane. The reason we introduce this extra object is that the images of nearby points in $ A_{n,c,\phi}$ will also be close to each other under certain conformal maps, in the sense of the following lemmas.
\begin{lemma}{{\rm (\cite{Tran}, $Lem.\;2.6$)}}
    There exist constants $\alpha>0$, and $c^\prime>0$, depending only on $c>0$ in the definition of the box $A_{n,c,\phi}$, such that for all $z_1,z_2\in A_{n,c,\phi}$ and conformal map $f:\mathbb{H}\to\mathbb{C}$, we have
    \begin{equation*}\begin{aligned}
        \abs{f^\prime(z_1)}\leq c^\prime\phi(n)^\alpha\vdot\abs{f^\prime(i\Im z_1)}
        \end{aligned}\end{equation*}
        and
        \begin{equation*}\begin{aligned}
        d_{\mathbb{H},hyp}(z_1,z_2)\leq c^\prime\log\phi(n)+c^\prime,
    \end{aligned}\end{equation*}
    where $ d_{\mathbb{H},hyp}(z_1,z_2)$ denotes the hyperbolic distance between $z_1$ and $z_2$ in $\mathbb{H}$.
\end{lemma}
\begin{lemma}{{\rm (\cite{map}, $Cor.\;1.5$)}}
    Suppose $f:\mathbb{H}\to\mathbb{C}$ is a conformal map, then for all $z_1,z_2\in\mathbb{H}$, we have
    \begin{equation*}
        \abs{f(z_1)-f(z_2)}\leq2\abs{(\Im z_1)f^\prime(z_1)}\vdot\exp\big(4d_{\mathbb{H},hyp}(z_1,z_2)\big).
    \end{equation*}
\end{lemma}
Therefore, it is natural that we want to show  $\{z,w\}\in A_{n,c,\phi}$ with proper parameters. Indeed, this is the case in the square-root interpolation (\cite{Tran}, $Lem.\;3.3$). The only non-trivial remark is the following.
\begin{remark}
   In the linear interpolation, from (\cite{Fredrikk}, $Sec.\;3.$) we know that for a typical linear force $\lambda(t)=t$ on the time interval $[0,\infty)$, the Loewner curve admits the form
   \begin{equation*}
       t\mapsto2-2\rho_t\cot\rho_t+2i\rho_t,
   \end{equation*}
   where $\rho_t$ increases monotonously from $\rho_0=0$ to $\rho_\infty=\pi$. Indeed, the Loewner curve of a general linear force $t\mapsto at+b$ requires some change of constant parameters depending on $a,b\in\mathbb{R}$, possibly with change in signs. Hence, we know the arc $\gamma^n_k:[0,\frac{1}{n}]\to\mathbb{H}\cup\{0\}$ corresponding to the
   piecewise-linear force $\lambda^n(t_k(t)+t)-\lambda^n(t_k(t))$ with $t\in[0,\frac{1}{n}]$ has an image which vertically stretches monotonically upward and horizontally either leftward or rightward. Hence, the images $\gamma^n_k([0,\frac{1}{n}])$ attains its maximal height at its tip $\gamma^n_k(\frac{1}{n})$, which justifies our choice of $z=\gamma_k(s)$.
\end{remark}
In this regard we could use $Lem.\;4.9$ and $Lem.\;4.10$ to give an upper bound to the second term in $(\ref{4.5})$. Now, let us turn our attention to the third term in $(\ref{4.9})$. This is actually a perturbation term: we need to measure the difference of one point in $\mathbb{H}$ under two conformal maps. In order to estimate the third term in $(\ref{4.9})$, we use the following result.
\begin{lemma}{{\rm (\cite{Cts}, $Lem.\;2.{3}$)}}
    Let $0<T<\infty$. Suppose $f^{(1)}_t$ and $f^{(2)}_t$ are the inverse map to the forward Loewner chain satisfying $(\ref{2.1})$ with driving force $W^{(1)}_t$ and $W^{(2)}_t$, respectively. Define $\epsilon\coloneqq\sup\limits_{s\in[0,T]}\abs{W^{(1)}_s-W^{(2)}_s}$. Then if $u=x+iy\in\mathbb{H}$, we have
    \begin{equation*}
        \abs{f^{(1)}_T-f^{(2)}_T}\leq\epsilon\exp\bigg\{\frac{1}{2}\bigg[\log\frac{I_{T,y}\abs{\partial_zf^{(1)}_T(u)}}{y}\log\frac{I_{T,y}\abs{\partial_zf^{(2)}_T(u)}}{y}\bigg]^\frac{1}{2}+\log\log\frac{I_{T,y}}{y}\bigg\}
    \end{equation*}
    where $I_{T,y}\coloneqq\sqrt{4T+y^2}.$ 
\end{lemma}
Applying $Lem.\;4.12$ to estimate the third term in $(\ref{4.9})$, combining the estimate to the first and the second terms,  on the event $F_{n,1}\cap F_{n,2}$, we see that
\begin{equation*}
    \sup\limits_{t\in[0,1]}\abs{\gamma(t)-\gamma^n(t)}\leq\frac{c_6(\log n)^{c_7}}{n^{(1-\sqrt{(1+\beta_1)/2})/2}}
\end{equation*}
for $\beta_1 \in (0,1).$
By definition of the event $F_n\in\Omega$, we know $F_{n,1}\cap F_{n,2}\subset F_n$. Moreover using the probability bounds on the events $F_{n,1}$ and $F_{n,2}$ in $(\ref{4.7})$, we obtain $Thm.\;4.8$.
\newline

\textbf{Conflict of interest statement.} The authors confirm that there is no conflict of interest.

\textbf{Acknowledgements.} We would like to acknowledge James Foster at Univ. Oxford for his help in sharing with us his simulation on $SLE_{8/3}$ and $SLE_{6}$ traces via the Ninomiya--Victoir splitting algorithm. We also acknowledge Lukas Schoug at Univ. Cambridge for his valuable comments and for reading preliminary versions of the manuscript. JC and VM  acknowledge the support of the NYU--ECNU Mathematical Institute at NYU Shanghai.

\bibliographystyle{plain}
\bibliography{literature}
\begin{spacing}{1}

\end{spacing}
\end{document}